\documentclass{amsart}
\usepackage{amsmath,amsthm,amsfonts,amssymb,indentfirst,xspace}
\usepackage{fancybox}
\usepackage{xcolor,soul}
\usepackage{graphicx,psfrag,subfigure}

\usepackage[initials,nobysame]{amsrefs}
\iffalse
\newcommand{\sidenote}[1]{\marginpar[\raggedleft\tiny #1]{\raggedright\tiny #1}}
\else
\newcommand{\sidenote}[1]{}
\fi

\newcommand{\torus}{\mathbb{T}}
\newcommand{\del}{\partial}
\newcommand{\lap}{\triangle}

\newcommand{\inv}{^{-1}}
\newcommand{\transpose}{^*}

\newcommand{\grad}{\nabla}
\newcommand{\gradt}{\grad\transpose}

\renewcommand{\epsilon}{\varepsilon}
\renewcommand{\leq}{\leqslant}
\renewcommand{\geq}{\geqslant}

\newcommand{\E}{\mathbb{E}}
\newcommand{\F}{\mathcal{F}}

\newcommand{\lhp}{\boldsymbol{\mathrm{P}}}

\newcommand{\holderspace}[2]{\ensuremath{C^{\ifx0#1{#2}\else{#1,#2}\fi}}}

\newif\iftextstyle
\textstyletrue
\everydisplay\expandafter{\the\everydisplay\textstylefalse}

\newcommand{\abs}[1]{\iftextstyle\lvert#1\rvert\else\left\lvert#1\right\rvert\fi}
\newcommand{\norm}[1]{\iftextstyle\lVert#1\rVert\else\left\lVert#1\right\rVert\fi}
\newcommand{\holdernorm}[3]{\norm{#1}_{\ifx0#2{#3}\else{#2,#3}\fi}}

\newcommand{\lpnorm}[2]{\norm{#1}_{\smash{L^{\!#2}_{\vphantom{h}}}\vphantom{L^{\!#2}}}}

\numberwithin{equation}{section}
\allowdisplaybreaks

\newtheorem{theorem}{Theorem}[section]
\newtheorem{lemma}[theorem]{Lemma}

\newtheorem*{theorem*}{Theorem}
\newtheorem*{lemma*}{Lemma}
\newtheorem*{proposition*}{Proposition}
\newtheorem*{corollary*}{Corollary}

\theoremstyle{definition}

\theoremstyle{remark}

\newtheorem*{remark*}{Remark}

\begin{document}
\title[A particle system for Navier-Stokes]{A Stochastic Lagrangian particle system for the Navier-Stokes equations}
\author{Alexei Novikov}
\email{anovikov@math.psu.edu}
\author{Karim Shikh Khalil}
\email{krs5562@psu.edu}
\address{Pennsylvania State University}
\begin{abstract}
 This work is based on a formulation of the incompressible Navier-Stokes equations
  developed by P. Constantin and G.Iyer, where the velocity
  field of a viscous incompressible fluid is written as the expected
  value of a stochastic process. If we take $N$ copies of
  the above process (each based on independent Wiener processes), and
  replace the expected value with the empirical mean, then it was shown 
  that the particle system for the Navier-Stokes equations does not dissipate all its energy as $t \to \infty$. 
  In contrast to the true (unforced) Navier-Stokes equations, which dissipate
all of its energy as $t \to \infty$. The objective of this short note is to describe  a resetting procedure that removes this deficiency. We prove that  if we repeat this resetting procedure often enough, then the new particle system for the Navier-Stokes equations dissipates all its energy.
\end{abstract}
\maketitle
\begingroup%
  \renewcommand\thefootnote{}%
  \footnote{%
    This material is based upon work partially supported by
        the National Science Foundation   DMS-1515187, and the Penn State Summer Research Experience for Undergraduates (REU) program.
  }
  \addtocounter{footnote}{-1}%
\endgroup%

\section{Introduction}

The Navier-Stokes equations 
\begin{align}
 \label{NS1} \partial_t u + (u \cdot \nabla)&u - \nu \Delta u+ \nabla p= 0,\\ 
\label{NS2} \nabla &\cdot u=0  
\end{align}
 describe the evolution of a velocity field $u$ of an incompressible fluid with kinematic
viscosity $ \nu >0 $, and where $p$ denotes the pressure. When the viscosity vanishes, we end up with the incompressible Euler equation:
\begin{align}
 \partial_t u + (u \cdot \nabla)&u+ \nabla p= 0,\\\label{E}
\nabla &\cdot u=0  
\end{align}
which describe the motion of an ideal incompressible fluid. 
 The mathematical theory of these equations have been extensively studied and the existence of regular solutions is still an open problem in  PDE theory~\cites{C,F}. We are interested in developing probabilistic techniques, that could help  solve this problem.

Probabilistic representations of solutions of partial differential equations as the expected value of functionals of stochastic processes date back to the work of Einstein, Feynman, Kac, 
and Kolmogorov in physics and mathematics. The 
Feynman-Kac formula is the most well-known example, which
has provided a link between linear parabolic partial differential equations and probability theory~\cite{KR, P}.
These stochastic representation methods have provided in some cases  tools to show existence and uniqueness of solutions to partial differential equations.
 For nonlinear  partial differential equations the earliest work was done by McKean~\cite{Mc}, where a probabilistic representation of the solution for the nonlinear Kolmogorov-Petrovsky-Piskunov equation was given. 
The theory for nonlinear partial differential equations, however, is far less understood.

The questions studied in this work are motivated by the development of a probabilistic formulation of~\eqref{NS1}-\eqref{NS2} proposed by P.Constantin and G.Iyer in~\cite{bblSLNS}. There, the Navier-Stokes equation 
is interpreted as the average of a 
stochastic perturbation of the Euler equation.  More specifically, a Weber formula is used to represent  the velocity  of the inviscid equation in terms of the particle trajectories of the inviscid equation without including time derivatives, then 
the classical Lagrangian trajectories are replaced by stochastic flows. Averaging these stochastic trajectories gives us the solution of~\eqref{NS1}-\eqref{NS2}.

In~\cite{bblParticleMethod}  G.Iyer and J.Mattingly  used a Monte-Carlo method to approximate the described probabilistic formulation.
They took  $N$ independent copies of the Wiener process and replaced the expected value in the above 
formalism  with the empirical mean,  $\frac{1}{N}$ times the sum over these $N$ independent copies (we review the details of this method in Section~\ref{PSR}). By the law of the large numbers it is natural to expect that any average could be replaced by its empirical mean: 
$1/N\sum_{i=1}^N X_i \approx \mathbb{E}(X)$, where $X$ and $X_i$, $i=1,\dots, N$ are i.i.d. 
It turns out that a straightforward approximation of this average by its empirical mean is not adequate here. 
It was shown in ~\cite{bblParticleMethod}  that in two dimensions the $N$-particle system for the Navier-Stokes equations does not dissipate all its energy as $t \rightarrow \infty $.  In contrast, the solution of the 
corresponding Navier-Stokes equations 
does dissipate all of its energy as  $t \rightarrow \infty $. The goal of this paper is to alleviate this deficiency of the particle system developed in ~\cite{bblParticleMethod} by modifying it, so that the modified particle system dissipates all of its energy. 

Our modification is inspired by~\cite{GN}, where G.Iyer and A.Novikov studied a particle system formulation for the Burgers equation. There we can find another example, where the particle system does not fully mimic
properties of the corresponding PDE. Namely, the viscous Burgers equation does not develop shocks, but
the corresponding $N$-particle system shocks almost surely in finite time. In order to remove these shocks, the authors~\cite{GN} considered a resetting procedure that prevented their 
 formation.  In this paper we propose another resetting procedure, such that the particle system for the Navier-Stokes equations dissipate all its energy. 
 
The particle system in~\cite{bblParticleMethod} does not completely dissipate its energy, because, roughly speaking, the gradients of the 
velocities for the $N$ particles become decorrelated with time. We reinforce correlation of these velocities and their gradients by 
resetting, and this allows complete dissipation of energy. When the resetting condition holds, then the particle system dissipates its energy exponentially. Once 
the resetting condition is not satisfied, we reset the particle system, and restart the procedure again. 
In addition to the exponential dissipation of energy, the resetting procedure itself adds more dissipation each time we average our data.
Our theorem states that if we keep repeating the resetting procedure, the particle system will dissipate all its energy.

We now highlight briefly some other probabilistic formulations for  the Navier-Stokes equations and related work.  The initial work on probabilistic representation of the Navier-Stokes equations was done by 
A.Chorin. It was shown in~\cite{Ch} that 
in two dimensions vorticity evolves according to Fokker-Planck type equation. Using this A.Chorin gave a probabilistic representation for the vorticity using random walks and a particle limit, and then related it to the velocity vector using the
Biot-Savart Law. Using a different approach Y.Le Jan and A.Sznitman  in~\cite{LS} developed a probabilistic representation of Navier-Stokes equations, where they used a backward-in-time branching process in Fourier space to express the 
velocity of the three dimensional viscous fluid as the average of a stochastic process. These works  do not allow to develop a self contained well-posedness theory. Here our motivation is to develop new probabilistic 
techniques that may help us establish regularity theory for partial differential equations of Fluid Dynamics. 
The stochastic representation in ~\cite{bblSLNS} allows such theory, and we, therefore, chose to work with it.

Building on the work done by P.Constantin and G.Iyer in~\cite{bblSLNS}, X.Zhang in ~\cite{XZ} gave a stochastic representation for the backward incompressible Navier-Stokes equation using stochastic Lagrangian 
path. Later, linking the work of X.Zhang in ~\cite{XZ} with  P.Constantin and G.Iyer  in~\cite{bblSLNS},   F.Delbaen, J.Qiu, and S.Tang in ~\cite{DQT} developed a coupled forward-backward 
stochastic differential system  in the space of fields for the incompressible Navier-Stokes equation in the whole space. Using probabilistic tools, they were able to obtain local uniqueness results for the 
forward-backward stochastic differential system. In addition, they were able to show the existence of global solutions for the case with small Reynolds number or when the dimension is two. 
We also mention~\cite{CC}, where F.Cipriano and A.Cruzeiro using the Brownian motions on the group of homeomorphisms on the torus,  established  a stochastic variational principle for the two dimensional  
Navier-Stokes equations. Moreover,  A.Cruzeiro and E.Shamarova in ~\cite{CS} formulated a connection between the Navier-Stokes equations and a system of 
forward-backward stochastic differential equations  on the group of volume-preserving diffeomorphisms of a flat torus.
 
This paper is organized as follows. In Section ~\ref{PSR} we describe the stochastic-Lagrangian representation of the Navier-Stokes equations and construct the particle system. In addition, we explain 
the resetting scheme which will then be used in Section \ref{Energy}. In Section \ref{Energy}, the main section of this paper, we study the energy of the Navier-Stokes's particle system and we use the resetting procedure and show that by repeating it often enough, the particle system for the Navier-Stokes equations dissipates all its energy.  
\section{The particle system and the resetting} \label{PSR} 

In this section we construct the particle system for the Navier-Stokes
equations based on stochastic Lagrangian trajectories. We begin by
describing a stochastic Lagrangian formulation of the Navier-Stokes
equations~\cite{bblParticleMethod}.

Let $B_t$ be a standard $2$ or $3$-dimensional Brownian motion on a torus $\torus$, and let 
$u_0$ be some given periodic, and divergence free $\holderspace{k}{\alpha}$ initial
data. Let $\E$ denote the expected value with respect to the Wiener
measure and $\lhp$ be the Leray-Hodge projection onto divergence free
vector fields. Consider the system of equations
\begin{align}
  \ dX_t(x) &= u_t ( X_t(x) ) \, dt + \sqrt{2\nu}\,dB_t, \qquad X_0(x) = x,\label{eNS} \\
   u_t &= \E\, \lhp\left[ (\gradt Y_t) (u_0 \circ Y_t)
  \right],\qquad Y_t = X_t\inv.  \label{eDivFree} 
\end{align}

Above $X_t$ is the stochastic flow of diffeomorphisms on $\torus$, and we denote $Y_t = X_t\inv$ to be the spatial inverse of $X_t$ for any given time $t \geq 0$. We
denote $\gradt Y_t$ to be the transpose of the Jacobian of $Y_t$.

It was shown in~\cite{bblSLNS} that if the initial data  $u_0 \in \holderspace{k}{\alpha}$ is a deterministic divergence free vector field  with $k \geq 2$ and if we impose periodic boundary conditions on $u_t$ and $X_t - \mathbb{I}$, then for a short time the system \eqref{eNS}-\eqref{eDivFree}  is equivalent to the incompressible Navier-Stokes equations, that is,
$u_t$ satisfies 
 \begin{equation}
 \partial_t u_t + (u_t \cdot \nabla) u_t - \nu \Delta u_t+ \nabla p= 0,  \hspace{0.1 cm}
\nabla \cdot u_t=0.  
\end{equation}
In the case when the viscosity is zero $\nu = 0$, the equations~\eqref{eNS}-\eqref{eDivFree} are the Lagrangian formulation for the incompressible Euler equation developed in~\cite{CA}.
Note that we need the law of the entire flow $X$ in order to compute $u$, this is due to the fact that the term $\nabla^{*}Y$ is present in~\eqref{eDivFree}.
In order to approximate the system \eqref{eNS}-\eqref{eDivFree}, we replace the flow $X_t$ with $N$ different copies $X^{i}_t$ where each one is driven independently by  a Wiener process $B^i_t$, $i=1,2, \dots, N$. 
Fix a (sufficiently large) $N$, and we end up with the following approximate system
\begin{align}\label{PrtlSys}
  dX^{i}_t &= u_t \left( X^{i}_t \right) \, dt + \sqrt{2\nu}\,dB^i_t,\quad  Y^{i}_t = \left( X^{i}_t \right)\inv,\\
 u_t &= \frac{1}{N} \sum_{i=1}^N u^{i}_t, \hspace{0.3 cm}
 u^{i}_t = \lhp[ (\gradt Y^{i}_t) u_0\circ Y^{i}_t], \label{SytmAvg}
\end{align}
with initial data $X_0(x) = x$. We impose periodic boundary
conditions on the initial data $u_0$, and the displacement $\lambda^{i}_t(x) =
X^{I}_t(x) - x$.

The following Lemma describes the evolution of the velocity of the particle system~\eqref{SytmAvg}
 as SPDE.

\begin{lemma}\label{lemma2.1} (Lemma 4.2 in~\cite{bblParticleMethod}) 
Let $u^{i}_t = \lhp[ (\gradt Y^{i}_t) u_{0} \circ Y^{i}_t]$ be the $i^\text{th}$ summand in \eqref{SytmAvg}. Then $u^{i}_t$ satisfies the SPDE
\begin{equation}\ \label{Eqn1}
  d u^{i}_t + \left[ (u_t \cdot \grad) u^{i}_t - \nu \lap
    u^{i}_t + (\gradt u_t) u^{i}_t +  \grad p^{i}_t\right] 
  dt   + \sqrt{2\nu} \grad u^{i}_t dB^i_t = 0, 
\end{equation}
and $u_t$ satisfies the SPDE
\begin{equation}
  d u_t + \left[ (u_t \cdot \grad) u_t - \nu \lap u_t + \grad
    p_t \right] \, dt + \frac{\sqrt{2\nu}}{N} \sum_{i=1}^N \grad
  u^{i}_t dB^i_t = 0. 
\end{equation}
\end{lemma}

In contrast to the true Navier-Stokes equations~\eqref{NS1}-\eqref{NS2}  the particle system~\eqref{PrtlSys}-\eqref{SytmAvg}, for any finite $N$,  may not 
dissipate all of its energy as $t \rightarrow \infty$.  In two dimensions this was proven in~\cite{bblParticleMethod}. 
 In this work, we propose to alleviate this deficiency by considering a resetting scheme, in which we start by  
 solving the system~\eqref{PrtlSys}-\eqref{SytmAvg} on the interval $(t_0, t_1] $, where the resetting  time $t_1$ is going to be specified later according to our proposed resetting condition $\eqref{times}$
below. Next, we average our data, replace the 
original initial data with $u^{N}_{t_1}$, and we restart the system~\eqref{PrtlSys}-\eqref{SytmAvg} for the next time interval using this new initial data. We keep repeating this procedure on each interval   
 $(t_m, t_{m+1}]$  for $m \in \mathbb{N}$.

The resetting criterion comes from comparison of the rate of change of the energies of the  true Navier-Stokes equation and the particle system. Namely, the rate of change of the energy of our particle system~\eqref{PrtlSys}-\eqref{SytmAvg} is (see Theorem \ref{thm3} below) 
as follows.
\begin{align} \label{RateEnergy2}
  \frac{1}{2\nu} \del_t \E \lpnorm{u_t}{2}^2  &=  -  \frac{1}{N^2} \hspace{0.15 cm}  \sum_{i \neq j }^N     \E \left[  \langle \nabla u^{j}_t, \nabla u^{i}_t \rangle \right].
\end{align}
 Observe
 that the rate of change of energy depends on the average of  inner products of the gradients of $N$ velocities.
  In contrast, for the true Navier-Stokes equation the rate of change of the energy 
\begin{align}\label{RateNSEnergy}
  \frac{1}{2\nu} \del_t  \lpnorm{u_t}{2}^2  &=  - \lpnorm{ \nabla u_t }{2}^2.
\end{align}
For large $N$ the right-hand sides of~\eqref{RateEnergy2} and~\eqref{RateNSEnergy} are essentially\footnote{ If $\nabla u^{j}_t=\nabla u^{i}_t$, then the  right-hand side of~\eqref{RateEnergy2}
is $ (N-1) N^{-1} \lpnorm{ \nabla u_t }{2}^2 \to  \lpnorm{ \nabla u_t }{2}^2$, as $N \to \infty$.}
 the same, if $\nabla u^{j}_t=\nabla u^{i}_t$ for all $i$ and $j$. This observation motivates our approach. 
 We will use resetting to keep the sum of the expected value of the inner products 
 \begin{align} \label{EnergyCondition}
 \sum_{i \neq j }^N \E \left[  \langle \nabla u^{j}_t, \nabla u^{i}_t \rangle \right] \geq \hspace{0.05 cm}  c N^2 \hspace{0.05 cm} \E \lpnorm{u_t}{2}^2 
 \end{align}
 on each interval 
$(t_m, t_{m+1}]$  where $m \in \mathbb{N}$ and for some constant $c >0$ that  does not depend on $N$. This will make the inner products  in~\eqref{RateEnergy2} to be positive and thus the rate of energy dissipation will be negative on each interval $(t_m, t_{m+1}]$. 
Therefore we consider the following resetting system
\begin{align}\label{ResttingPrtlSys1}
  dX^{i}_{t } &= u_t \left( X^{i}_{t } \right) \, dt + \sqrt{2\nu}\,dB^i_t, \hspace{0.2 cm}  X^{i}_{t_m }(x_0)=x_0,~
  Y^{i}_{t }= \left( X^{i}_{t } \right)\inv,\\
 u_{t } &= \frac{1}{N} \sum_{i=1}^N u^{i}_t, \hspace{0.2 cm}   u^{i}_t = \lhp\left[ (\gradt Y^{i}_{t }) (u_{t_m} \circ Y^{i}_{ t }) \right],  \hspace{0.2 cm}  \text{for} \hspace{0.2 cm} 
 t \in (t_m, t_{m+1}],  \label{ResttingSytmAvg1}
\end{align}
where $m \in \mathbb{N}$, the set of non-negative integers, $t_0=0$, and the resetting times are defined recursively 
 \begin{equation} \label{times}
t_{m} = \inf \Big\{ t > t_{m-1}: \sum_{i \neq j }^N  \hspace{0.15 cm}     \E \left[  \langle \nabla u^i_t, \nabla u^j_t \rangle \right ]  < (1- \epsilon ) N(N-1)   \lpnorm{\nabla u_{t_{m-1} }}{2}^2     \Big\}. 
\end{equation}
for some positive fixed  $\epsilon <1$.
We say we reset the system at every $t =  t_m$, because we treat $u_{ t_m}$ as initial conditions for each of the 
intervals $t \in [ t_m, t_{m+1} )$. 

\begin{theorem}\label{thm3}
 Suppose we are in two dimensions.
 Let the initial condition $u_{0}$ be a  
  $\F_{0}$-measurable, periodic mean zero function such that the
  norm $\holdernorm{u_{0}}{1}{\alpha}$, $\alpha >0$ is almost surely
  bounded. If we let $\{t_m\}_{m=0}^{\infty}$ to be the sequence of resetting times defined in \eqref{times}, then the particle system with resetting   \eqref{ResttingPrtlSys1}-\eqref{ResttingSytmAvg1}  dissipates all its energy.
\end{theorem}

We remark that the particle system with resetting ~\eqref{ResttingPrtlSys1}-\eqref{ResttingSytmAvg1} dissipates its energy using two mechanisms. It dissipates energy exponentially anytime 
the inequality~$\eqref{EnergyCondition}$ holds, and  it dissipates energy when we average our data each time we reset our system. 
We also want to highlight that the manner in which we defined our resetting times in ~\eqref{times} causes the length of time increments $\delta_m= t_m-t_{m-1}$
 to vary among resetting intervals. This gives rise to the case that if the sequence of time increments $\delta_m $ decays to zero too fast, then the limit of the sequence of resetting times $t_m \rightarrow T$, for some constant 
 time $T$. Thus, we have to consider two cases. First case is when the limit of sequence of resetting times $t_m \rightarrow \infty$, as $m \to \infty$. In this case we show that the energy of the particle system with 
 resetting~\eqref{ResttingPrtlSys1}-\eqref{ResttingSytmAvg1} dissipates its energy mainly exponentially. The second case is when the limit of the sequence 
 of resetting times $t_m \rightarrow T$, for some finite time $T$. 
 In this case, we show that the system ~\eqref{ResttingPrtlSys1}-\eqref{ResttingSytmAvg1} dissipates its energy mainly by averaging  each time we reset.

\section{ Energy decay by resetting} \label{Energy}

\begin{proof}[Proof of Theorem~\ref{thm3}]

By Lemma \ref{lemma2.1}, 
$u^{i}_t$ satisfies the SPDE
\begin{equation*}
  d u^{i}_t + \left[ (u_t \cdot \grad) u^{i}_t - \nu \lap
    u^{i}_t + (\gradt u_t) u^{i}_t +  \grad p^{i}_t\right] 
  dt   + \sqrt{2\nu} \grad u^{i}_t dB^i_t = 0, 
\end{equation*}
on each interval $t \in  § (t_m, t_{m+1}]$.
In two dimensions, the vorticity  $\omega^i_t = \nabla \times u^i_t$ solves 
\begin{eqnarray*}
d \omega^i_t+ [ (u_t \cdot \nabla) \omega^i -\nu \Delta \omega^i_t] dt +\sqrt{2 \nu} \nabla \omega^i_t dB^i_t=0.
\end{eqnarray*}
By  It\^o's formula, we  obtain 
\begin{multline*}
  \frac{1}{2} d \abs{\omega^i_t}^2 + \omega^i_t \cdot \left[ (u_t \cdot \grad) \omega^i_t  - \nu
    \lap \omega^i_t   \right] \, dt 
  + \sqrt{2 \nu} \omega^i_t \cdot (\grad \omega^i_t\, dB^i_t) -
  \nu \abs{\grad \omega^i_t}^2 \, dt=0
\end{multline*}
Integrating in space and using the fact that $u^i_t$ and $u_t$ are divergences free, 
we have   $ d \lpnorm{\omega^i_t}{2}^2=0$ for all  $ t \in [t_m, t_{m+1})$. Thus the norm of all the vorticities is preserved on such  time-intervals. 
Since   $\lpnorm{\omega^i_t}{2}^2 =\lpnorm{\nabla u^i_t}{2}^2$ we have
\[
\lpnorm{\nabla u_{t_m}}{2}^2 =\lpnorm{\nabla u^1_t}{2}^2=\lpnorm{\nabla u^2_t}{2}^2=\dots=\lpnorm{\nabla u^N_t}{2}^2
\]
for $t \in  § (t_m, t_{m+1}]$.

Using Lemma \ref{lemma2.1} and  It\^o's formula, we also have  
\begin{multline*}
  \frac{1}{2} d \abs{u_t}^2 + u_t \cdot \left[ (u_t \cdot \grad) u_t - \nu
    \lap u_t + \grad p \right] \, dt 
  + \frac{\sqrt{2 \nu}}{N} \sum_{i=1}^N u_t \cdot (\grad u^i_t\, dB^i_t) \\=
  \frac{\nu}{N^2} \sum_{i=1}^N \abs{\grad u^i_t}^2 \, dt
\end{multline*}
Integrating in space and taking expected values, we obtain
\begin{align} \label{RateIto1}
  \frac{1}{2\nu} \del_t \E \lpnorm{u_t}{2}^2 &= \E \left[
    \frac{1}{N^2} \sum_{i=1}^N \lpnorm{\grad u^i_t}{2}^2 -
    \lpnorm{\grad u_t}{2}^2 \right]\\  \label{RateIto2}
  &= \E \left[ \frac{1}{N^2} \sum_{i=1}^N \lpnorm{\grad u^i_t}{2}^2 - \frac{1}{N^2} \bigg [ \sum_{i=1}^N \| \nabla u^i_t \|^2_{L^2}+   \sum_{i \neq j }^N  \langle \nabla u^j_t, \nabla u^i_t \rangle \bigg ] \right ]
\end{align}
This simplifies to
\begin{align} \label{RateEnergy3}
  \frac{1}{2\nu} \del_t \E \lpnorm{u_t}{2}^2 &= -  \frac{1}{N^2} \hspace{0.15 cm}  \sum_{i \neq j }^N   \E \left[  \langle \nabla u^j_t, \nabla u^i_t \rangle \right ]  
\end{align}

\textit{Case I:} $\lim_{m \to \infty} t_m \rightarrow \infty $. In this case the sequence of resetting times goes to infinity. Using resetting, we have 
\begin{align} \label{ResettingCondtion}
\sum_{i \neq j }^N  \hspace{0.15 cm}     \E \left[  \langle \nabla u^i_t, \nabla u^j_t \rangle \right ]    \geq (1- \epsilon ) N(N-1) \lpnorm{\nabla u_{t_{m} }}{2}^2  
\end{align}  
for all $t \in [t_m, t_{m+1})$. Thus, using $\eqref{ResettingCondtion} $ we can obtain the following estimate on the rate of change of energy $\eqref{RateEnergy3}$ 
\[
  \frac{1}{2\nu} \del_t \E \lpnorm{u_t}{2}^2   \leq - (1- \epsilon )\frac{N-1}{N} \lpnorm{\nabla u_{t_m }}{2}^2
  \]
\[
\leq -C(1- \epsilon )\frac{N-1}{N} \E \lpnorm{u_{t_m }}{2}^2 \leq -C(1- \epsilon )\frac{N-1}{N} \E \lpnorm{u_{t }}{2}^2
\]
for some constant  $C>0$ that arises from using   Poincar\'{e}'s  inequality.  Using the Gronwall's inequality, we obtain exponential dissipation of energy.
 
\textit{Case II:} $ \lim_{m \to \infty} t_m \rightarrow T $. In this case the sequence of resetting times converges to a finite time $T$. At every resetting time we have
\begin{align*}
 \omega_{t_{m}   }=   \frac{1}{N} \hspace{0.15 cm}  \sum_{i=1 }^N  \omega^i_{t_{m}}. 
\end{align*}
Thus,  
\[
\E \lpnorm{ \omega_{t_{m}}}{2}^2=  \frac{1}{N^2}    \sum_{i=1 }^N \E \lpnorm{\omega^i_{t_{m}}}{2}^2   + \frac{1}{N^2} \sum_{i \neq j }^N \E \left[   \langle    \omega^i_{t_{m}} ,   \omega^j_{t_{m}}  \rangle \right]
\]
\[
= \frac{1}{N}       \lpnorm{\omega_{t_{m-1}}}{2}^2   + \frac{1}{N^2} \sum_{i \neq j }^N \E \left[   \langle    \omega^i_{t_{m}} ,   \omega^j_{t_{m}}  \rangle \right]. \label{AverageL2norm}
 \]
 Since  $\langle \nabla u^i_{t_{m}}, \nabla u^j_{t_{m}} \rangle =\langle \omega^i_{t_{m}}, \omega^j_{t_{m}} \rangle$,  we have 
\begin{align} \label{ResettingCondtion2}
\sum_{i \neq j }^N  \hspace{0.15 cm}     \E \left[  \langle \omega^i_{t_{m}}, \omega^j_{t_{m}} \rangle \right ]    \leq  (1- \epsilon ) N(N-1)  \lpnorm{\omega_{t_{m-1} }}{2}^2  
\end{align}  
\begin{align*} 
\E \lpnorm{ \omega_{t_{m}}}{2}^2 \leq \big( 1 -\alpha \big) \lpnorm{\omega_{t_{m-1} }}{2}^2, \alpha =1 -\epsilon+\frac{\epsilon}{N} <1.  
 \end{align*}
 Iterating over $m$, we have 
 \begin{align*} 
\E \lpnorm{ \omega_{t_{m}}}{2}^2 \leq \big( 1 -\alpha \big)^m \lpnorm{\omega_{0 }}{2}^2, \hspace{0.1 cm} \text{where} \hspace{0.1 cm}  \omega_0 = \nabla \times u_0.  
 \end{align*}
 Thus, if $\lim_{m \to \infty} t_m \rightarrow T$, for some finite time $T$,  this means we are going to reset our particle system a countable number of times. Hence, 
  \begin{align*} 
\E \lpnorm{ \omega_{t_{m}}}{2}^2 \leq \big( 1 -\alpha \big)^m \lpnorm{\omega_{0}}{2}^2   \rightarrow 0\hspace{0.2 cm } \text{as} \hspace{0.2 cm } m \rightarrow \infty. 
 \end{align*}
Thus, the particle system dissipates all its energy in a finite time $T$.

\end{proof}


\begin{thebibliography}{99}





 \bibitem{Ch} A.Chorin \textit{Numerical study of slightly viscous flow}, J. Fluid Mech. 57 (1973), no. 4, 785--796.



\bibitem{CC} 
 F. Cipriano and A.B. Cruzeiro:\textit{ Navier-Stokes equation and Diffusions on the Group of Homeomorphisms of the Torus},  Commun. Math. Phys., 275, 255-269(2007). 









\bibitem{CS} 
A. Cruzeiro, E. Shamarova:\textit{Navier-Stokes equations and forward-backward SDEs on the group of diffeomorphisms of a torus},  http://arxiv.org/abs/0807.0421.  



\bibitem{CA} 
P. Constantin,
\textit{An Eulerian-Lagrangian approach for incompressible fluids: local theory},   
. Math. Soc. 14 (2001), no. 2, 263–278 (electronic).



\bibitem{C} 
P. Constantin,
\textit{Some open problems and research directions in the mathematical study of fluid
dynamics},  Mathematics unlimited—2001 and beyond, Springer, Berlin, 2001, pp. 353â 360.

\bibitem{bblSLNS} 
P.Constantin, G.Iyer (2006)
\textit{A stochastic Lagrangian representation of the three-dimensional
  incompressible Navier--Stokes equations}, 
Comm. Pure Appl. Math 61 330–
345. MR2376844

\bibitem{DQT} 
 F. Delbaen, J. Qiu, S. Tang, \textit{Forward-backward stochastic differential systems associated to Navier-Stokes equations in the whole space}, arXiv:1303.5329v2 [math-ph].

\bibitem{EMPS} 
 R. Esposito, R. Marra, M. Pulvirenti, C. Sciarretta \textit{A stochastic Lagrangian picture for the three dimensional Navier-Stokes equation},  Comm. Partial Differential Equations. 13 (1988), no. 12, 1601– 1610


\bibitem{F} 
 C. L. Fefferman,
\textit{Existence and smoothness of the Navier-Stokes equation},  The millennium
prize problems, Clay Math. Inst., Cambridge, MA, 2006, pp. 57--67.





\bibitem{bblSPerturb} 
G.Iyer (2006)
\textit{A stochastic perturbation of inviscid flows}, 
Comm. Math. Phys. 266 631-645.
MR2238892



 


\bibitem{bblParticleMethod} 
G.Iyer, J.Mattingly
\textit{A stochastic-Lagrangian particle system for the Navier--Stokes
  equations}, 
Nonlinearity 21 2537-2553. MR2448230

\bibitem{GN}
G. Iyer, A. Novikov, \textit{The regularizing effects of resetting in a particle system for the Burgers equation}, The Annals of Probability 2011, Vol. 39, No. 4, 1468--1501



\bibitem{KR} 
N. V. Krylov and B. L. Rozovski
\textit{   Stochastic partial differential equations and diffusion processes}, Uspekhi Mat. Nauk 37 (1982), no. 6(228), 75--95 (Russian).









\bibitem{Mc}
  H.P. McKean, \textit{Application of Brownian motion to the equation of Kolmogorov, Petrovskii, and Piskunov}, Commun. Pure Appl. Math. 28 (3) (1975) 323--331.








\bibitem{P} 
J. Pedlosky, 
\textit{Geophysical Fluid Dynamics},  Springer-Verlag, 1982.



 


\bibitem{S} 
 A.-S. Sznitman, 
\textit{Topics in propagation of chaos},    Ecole d' ´ Et´e de Probabilit´es de Saint-Flour ´
XIX—1989, Lecture Notes in Math., vol. 1464, Springer, Berlin, 1991, pp. 165--251.
 









\bibitem{LS}   
Le Jan Y.; Sznitman A. S.  \textit{Stochastic cascades and 3-dimensional Navier-Stokes equations}, 
Probab. Theory Related Fields 109 (1997), no. 3, 343--366.




\bibitem{XZ} 
X. Zhang, \textit{
A stochastic representation for backward incompressible Navier-Stokes equations},  Prob. Theory
and Rela. Fields, Volume 148, Numbers 1-2, 305-332 (2010).

\bibitem{XZ2} 
X. Zhang, \textit{Stochastic Lagrangian Particle Approach to Fractal Navier-Stokes Equations}, Communnications in Mathmatical Physics, 311, 1, 133-155 (2012).




\end{thebibliography}
\end{document}